\documentclass[10pt]{article}

\usepackage[english]{babel}
\usepackage[utf8]{inputenc}
\usepackage[T1]{fontenc}

\usepackage{amsmath, amsthm, amssymb, mathtools}

\usepackage[shortlabels]{enumitem}

\usepackage{cases}

\usepackage[a4paper]{geometry}

\AtBeginDocument{\def\MR#1{}}

\usepackage{hyperref}

\providecommand{\keywords}[1]{\textbf{Keywords.} #1}
\providecommand{\MSC}[1]{\textbf{AMS Subject Classifications.} #1}

\newtheorem{theorem}{Theorem}[section]

\newtheorem{lemma}[theorem]{Lemma}

\newtheorem{example}[theorem]{Example}

\theoremstyle{definition}
\newtheorem{definition}[theorem]{Definition}
\newtheorem{remark}[theorem]{Remark}

\renewcommand\epsilon{\varepsilon}
\renewcommand\mapsto{\longmapsto}

\usepackage{color}

\usepackage{xspace}

\newcommand{\R}{\field{R}\xspace}
\newcommand{\N}{\field{N}\xspace}

\newcommand{\field}[1]{\ensuremath{\mathbb{#1}}}

\newcommand{\ens}[1]{ \left\{#1\right\} }

\newcommand\diag{\mathrm{diag} \xspace}

\newcommand{\Topt}{T_{\mathrm{opt}}\left(\Lambda\right)}
\newcommand{\Tunif}{T_{\mathrm{unif}}\left(\Lambda\right)}

\newcommand\pt[1]{\frac{\partial #1}{\partial t}}
\newcommand\px[1]{\frac{\partial #1}{\partial x}}
\newcommand\pxi[1]{\frac{\partial #1}{\partial \xi}}
\newcommand\pas[1]{\frac{\partial #1}{\partial s}}

\newcommand\Tau{\mathcal{T}}

\newcommand\ssin{s^{\mathrm{in}}}
\newcommand\ssout{s^{\mathrm{out}}}

\newcommand\loc{{\mathrm{loc}}}

\newcommand\xmax{x_\Lambda}

\newcommand\Xm[2]{\ell_{#1}\left(#2\right)}

\title{Null controllability and finite-time stabilization in minimal time of one-dimensional first-order $2 \times 2$ linear hyperbolic systems}

\author{
Long Hu\thanks{School of Mathematics, Shandong University, Jinan, Shandong 250100, China.  E-mail: \texttt{hul@sdu.edu.cn}}
\and
Guillaume Olive\thanks{Institute of Mathematics, Jagiellonian University, Lojasiewicza 6, 30-348 Krakow, Poland. E-mail: \texttt{math.golive@gmail.com} or \texttt{guillaume.olive@uj.edu.pl}}
}

\date{\today}

\begin{document}

\maketitle

\begin{abstract}
The goal of this article is to present the minimal time needed for the null controllability and finite-time stabilization of one-dimensional first-order $2 \times 2$ linear hyperbolic systems.
The main technical point is to show that we cannot obtain a better time.
The proof combines the backstepping method with the Titchmarsh convolution theorem.
\end{abstract}


\vspace{0.2cm}
\MSC{35L40, 93B05, 93D15, 45D05}


\section{Introduction and main result}

\subsection{Problem description}

In this paper we are interested in the characterization of the minimal time needed for the controllability of the following class of one-dimensional first-order $2 \times 2$ linear hyperbolic systems:

\begin{equation}\label{syst}
\begin{dcases}
\pt{y_1}(t,x)+\lambda_1(x) \px{y_1}(t,x)=a(x)y_1(t,x)+b(x)y_2(t,x), \\
\pt{y_2}(t,x)+\lambda_2(x) \px{y_2}(t,x)=c(x)y_1(t,x)+d(x)y_2(t,x), \\
\begin{aligned}
& y_1(t,1)=u(t), && y_2(t,0)=0,  \\
& y_1(0,x)=y_1^0(x), && y_2(0,x)=y_2^0(x),
\end{aligned}
\end{dcases}
\quad t \in (0,+\infty), \, x \in (0,1).
\end{equation}

Such systems appear in linearized versions of various physical models of balance laws, see e.g. \cite[Chapter 1]{BC16}.
For instance, the telegrapher equations of Heaviside form a linear system of the form \eqref{syst} for some parameters (see e.g. \cite[Section 1.2 and (1.20)]{BC16} with $-1+\lambda R_0 C_\ell=0$).

In \eqref{syst}, $(y_1(t,\cdot),y_2(t,\cdot))$ is the state at time $t$, $(y_1^0,y_2^0)$ is the initial data and $u(t)$ is the control at time $t$.
We assume that the speeds $\lambda_1,\lambda_2 \in C^{0,1}([0,1])$ are such that
$$
\lambda_1(x)<0<\lambda_2(x), \quad \forall x \in [0,1].
$$
Finally, $a,b,c,d \in L^{\infty}(0,1)$ couple the equations of the system inside the domain (the matrix $\begin{pmatrix} a & b \\ c & d \end{pmatrix}$ will also be referred in the sequel to as the internal coupling matrix).

We recall that the system \eqref{syst} is well-posed: for every $u \in L^2_{\loc}(0,+\infty)$ and $(y^0_1, y^0_2) \in L^2(0,1)^2$, there exists a unique solution $(y_1,y_2) \in  C^0([0,+\infty);L^2(0,1)^2)$ to the system \eqref{syst}.
By solution we mean ``solution along the characteristics'' or ``broad solution'' (see e.g. \cite[Appendix A]{CHOS21}).
The same statement remains true if, in the boundary condition at $x=1$, $u$ is replaced by
\begin{equation}\label{feedback law}
u(t)=\int_0^1 \left(f_1(\xi)y_1(t,\xi)+f_2(\xi)y_2(t,\xi)\right) \,d\xi,
\end{equation}
for any $f_1,f_2 \in L^{\infty}(0,1)$.
The relation \eqref{feedback law} is called the ``feedback law''.

Let us now introduce the notions of controllability that we are interested in:

\begin{definition}\label{def stab FT}
Let $T>0$.
We say that the system \eqref{syst} is:
\begin{itemize}
\item
\textbf{finite-time stable with settling time $T$} if, for every $y^0_1, y^0_2 \in L^2(0,1)$, the corresponding solution to the system \eqref{syst} with $u=0$ satisfies
\begin{equation}\label{final cond}
y_1(T,\cdot)=y_2(T,\cdot)=0.
\end{equation}

\item
\textbf{finite-time stabilizable with settling time $T$} if there exist $f_1,f_2 \in L^{\infty}(0,1)$ such that, for every $y^0_1, y^0_2 \in L^2(0,1)$, the corresponding solution to the system \eqref{syst} with $u$ given by \eqref{feedback law} satisfies \eqref{final cond}.

\item
\textbf{null controllable in time $T$} if, for every $y^0_1, y^0_2 \in L^2(0,1)$, there exists $u \in L^2_{\loc}(0,+\infty)$ such that the corresponding solution to the system \eqref{syst} satisfies \eqref{final cond}.
\end{itemize}

\end{definition}

Obviously, finite-time stability implies finite-time stabilization, which in turn implies null controllability.

\begin{remark}\label{rem counterexample}
As we are trying to bring the solution of the system \eqref{syst} to the state zero, let us first mention that, in general, $u=0$ does not work.
Not only this, but in fact any static boundary output feedback laws, that is of the form $u(t)=k y_2(t,1)$ with $k \in \R$, does not work either in general.
A simple example is provided by the following $2 \times 2$ system with constant coefficients (see also \cite[Section 5.6]{BC16} when $y_2(t,0)=y_1(t,0)$):
\begin{equation}\label{syst ex}
\begin{dcases}
\pt{y_1}(t,x)-\px{y_1}(t,x)=\pi y_2(t,x), \\
\pt{y_2}(t,x)+ \px{y_2}(t,x)=\pi y_1(t,x), \\
\begin{aligned}
& y_1(t,1)=k y_2(t,1), && y_2(t,0)=0,  \\
& y_1(0,x)=y^0_1(x), && y_2(0,x)=y^0_2(x),
\end{aligned}
\end{dcases}
\quad t \in (0,+\infty), \, x \in (0,1).
\end{equation}
Indeed, for this system we can always construct a smooth initial data $(y^0_1,y^0_2)$ which is an eigenfunction of the operator associated with \eqref{syst ex} and whose corresponding eigenvalue $\sigma$ is a positive real number, which makes the system \eqref{syst ex} exponentially unstable.
This can be done as follows.
We take
$$y^0_1(x)=\frac{1}{\pi}\left(\sigma y^0_2(x)+ \px{y^0_2}(x)\right),$$
(so that the second equation in \eqref{syst ex} will always be satisfied) and
\begin{itemize}
\item
If $k<1+1/\pi$, then we take $\sigma=\pi\sqrt{1-\theta^2}$ and $y^0_2(x)=\sin(\theta \pi x)$, where $\theta \in (0,1)$ is any solution to the equation $\sqrt{1-\theta^2}+\theta \cot(\theta \pi)=k$.

\item
If $k=1+1/\pi$, then we take $\sigma=\pi$ and $y^0_2(x)=\pi x$.

\item
If $k>1+1/\pi$, then we take $\sigma=\pi\sqrt{1+\theta^2}$ and $y^0_2(x)=2\sinh(\theta \pi x)$, where $\theta>0$ is any solution to the equation $\sqrt{1+\theta^2}+\theta \coth(\theta \pi)=k$.
\end{itemize}

\end{remark}

The goal of this work is to establish a necessary and sufficient condition on the time $T$ for the system \eqref{syst} to be null controllable in time $T$ (resp. finite-time stabilizable with settling time $T$).

Let us now introduce some notations that will be used all along the rest of this article.
Let $\phi_1,\phi_2 \in C^{1,1}([0,1])$ be the increasing functions defined for every $x \in [0,1]$ by
\begin{equation}\label{def phii}
\phi_1(x)=\int_0^x \frac{1}{-\lambda_1(\xi)} \,d\xi,
\quad
\phi_2(x)=\int_0^x \frac{1}{\lambda_2(\xi)} \,d\xi.
\end{equation}
We then denote by
$$
T_1(\Lambda)
=\phi_1(1)
=\int_0^1 \frac{1}{-\lambda_1(\xi)} \, d\xi,
\quad
T_2(\Lambda)
=\phi_2(1)
=\int_0^1 \frac{1}{\lambda_2(\xi)} \, d\xi.
$$
Finally, we set
\begin{equation}\label{def Topt Tunif}
\Topt=\max\ens{T_1(\Lambda), T_2(\Lambda)},
\quad
\Tunif=T_1(\Lambda)+T_2(\Lambda).
\end{equation}
The naming of the notations in \eqref{def Topt Tunif} will be explained in Remark \ref{rem naming} below.

\subsection{Literature}

Boundary null controllability and stabilization of hyperbolic systems of balance laws have attracted numerous attention of both mathematicians and engineers during the last decades.
In the pioneering work \cite{Rus78}, the author established the null controllability of general $n\times n$ coupled linear hyperbolic systems of the form \eqref{syst} in a control time that is given by the sum of the two largest times from the states convecting in opposite directions (\cite[Theorem 3.2]{Rus78}).
It was also observed that this time can be shorten in some cases (\cite[Proposition 3.4]{Rus78}), and the problem to find the minimal control time for hyperbolic partial differential equations (PDEs) was then raised (\cite[Remark p. 656]{Rus78}).

For systems of linear conservation laws (i.e. when no internal coupling matrix is present in the system), this problem was completely solved few years later in \cite{Wec82}, where the minimal control time has been characterized in terms of the boundary coupling matrix, that is the matrix coupling the equations at the boundary on the uncontrolled side.
For systems of balance laws,  the story is far from over.
A first improvement of the control time of \cite{Rus78} was recently obtained in \cite{CN19} thanks to the introduction of some rank condition on the boundary coupling matrix.
However, this was first done for some generic internal coupling matrices or under rather stringent conditions (\cite[Theorem 1.1 and 1.5]{CN19}).
The same authors were then able to remove some of these restrictions in \cite{CN19-pre}.
For the present paper it is especially important to emphasize that the new time introduced in \cite{CN19,CN19-pre} is only shown to be sufficient for the null controllability in these works.
On the other hand, the minimal control time needed to achieve the exact controllability property (that is when we want to reach any final data and not only zero), was completely characterized in \cite[Theorem 1.9]{HO19} by a simple and calculable formula.
It is also pointed out that null and exact controllability are equivalent properties if the boundary coupling matrix has a full row rank.
For quasilinear systems, it has been shown in \cite[Theorem 3.2]{Li10} that the time of \cite{Rus78} yields the (local) exact controllability of such systems if the boundary coupling matrix has a full row rank in a neighborhood of the state zero.
For homogeneous quasilinear systems, a smaller control time was then obtained in \cite[Theorem 1.1]{Hu15}.

Concerning now the stabilization property, the first works seem \cite{GL84,Qin85} for the exponential stabilization of homogeneous quasilinear hyperbolic systems in a $C^1$ framework by using the method of characteristics.
To the best of our knowledge, the weakest sufficient condition using this technique can be found in \cite[Theorem 1.3, p. 173]{Li94}.
This condition was then improved in \cite[Theorem 2.3]{CBdAN08}  in a $H^2$ framework thanks to the construction of an explicit strict Lyapunov function.
In all the previous references, the feedback laws were static boundary output feedback laws (that is, depending only on the state values at the boundaries).
However, due to the locality of such kind of feedback laws, these two strategies may not be effective to deal with general systems of balance laws (\cite[Section 5.6]{BC16} and Remark \ref{rem counterexample}).
Another method was then used to address this problem, the backstepping method.
For PDEs, this method now consists in transforming our initial system into another system - called target system - for which the stabilization properties are simpler to study.
The transformation used is usually a Volterra transformation of the second kind.
One can refer to the tutorial book \cite{KS08} to design boundary feedback laws stabilizing systems modeled by various PDEs and to the introduction of \cite{CHOS21} for a complementary state of the art on this method.
This technique turned out to be a powerful tool to stabilize general coupled hyperbolic systems, moreover in finite time.
In \cite{CVKB13} the authors adapted this technique to obtain the first finite-time stabilization result for $2\times 2$ linear hyperbolic system.
This method was then developed, notably with a more careful choice of the target system, to treat $3 \times 3$ systems in \cite{HDM15} and then to treat general $n\times n$ systems in \cite{HDMVK16,HVDMK19}.
However, the control time obtained in these works was larger than the one in \cite{Rus78} and it was only shown in \cite{ADM16, CHO17} that we can stabilize with the same time as the one of \cite{Rus78}.
These works have recently been generalized to time-dependent systems in \cite{CHOS21}.
Finally, let us also mention the two recent works \cite{CN20-pre1, CN20-pre2} concerning the finite-time stabilization of homogeneous quasilinear systems, with the same control time as in \cite{CN19,CN19-pre}.

In spite of quite a number of contributions dealing with these two problems (controllability and stabilization), we see that there are no references concerning the optimality of the control time for systems of linear balance laws with spatial-varying internal coupling matrix, especially when null and exact controllability are not equivalent, so that the results in \cite{CVKB13, HO19} cannot be considered.
This is of course a nontrivial task and it requires the addition of new techniques as we shall see below.
The goal of this article is to fill this gap, at least for $2 \times 2$ systems.
We will provide an explicit formula of the minimal control time for any $2\times 2$ system of linear balance laws with spacial-varying internal coupling matrix.
We will see that one of the main differences between null and exact controllability is that such a critical time is sensitive to the behavior of the internal coupling matrix for the null controllability, whereas it is known to never be the case for the exact controllability (\cite{CVKB13, HO19}).

\subsection{Main result and comments}

The important quantity in the present work is the following:

\begin{definition}\label{def xf}
For $\epsilon>0$ and a function $f:(0,\epsilon) \longrightarrow \R$, we denote by
$$\Xm{\epsilon}{f}=
\begin{dcases}
\sup I_\epsilon(f) & \mbox{ if } I_\epsilon(f) \neq \emptyset, \\
0 & \mbox{ otherwise, }
\end{dcases}
$$
where $I_\epsilon(f)=\ens{\ell \in (0,\epsilon) \quad \middle| \quad f=0 \, \text{ a.e. in } (0,\ell)}$.
\end{definition}

The quantity $\Xm{\epsilon}{f}$ is the length of the largest interval of the form $(0,\ell)$ where the function $f$ vanishes.

\begin{example}\label{examples}
~
\begin{enumerate}[(E1)]
\item\label{ex 1}
The simplest example of function $f$ with $\Xm{\epsilon}{f}=\ell$ ($\ell \in [0,\epsilon]$) is obviously the step function
$$
f(x)
=
\begin{dcases}
0 & \mbox{ if } x\leq \ell, \\
1 & \mbox{ if } x>\ell.
\end{dcases}
$$

\item\label{ex 2}
If $f \in C^k([0,\epsilon))$ ($k \in \N$) and satisfies $f^{(k)}(0) \neq 0$, then $\Xm{\epsilon}{f}=0$.
In particular, if $f$ has an analytic extension in a neighborhood of $x=0$, then $\Xm{\epsilon}{f}=0$.

\item\label{ex 3}
An example of smooth function $f$ with $\Xm{\epsilon}{f}=0$ but that does not satisfy the previous conditions is
\begin{equation}\label{bump f}
f(x)
=
\begin{dcases}
0 & \mbox{ if } x \leq 0, \\
\exp\left(-\frac{1}{x}\right) & \mbox{ if } x>0.
\end{dcases}
\end{equation}

\end{enumerate}

\end{example}

The main result of this article is the following complete characterization of the controllability properties of the system \eqref{syst}:
\begin{theorem}\label{main thm}
Let $T>0$.
\begin{enumerate}[(i)]
\item\label{item NC}
If the system \eqref{syst} is null controllable in time $T$, then necessarily
\begin{equation}\label{min time syst}
T \geq \max\ens{\Topt, \quad \int_{\Xm{\xmax}{c}}^1 \left(\frac{1}{-\lambda_1(\xi)}+\frac{1}{\lambda_2(\xi)}\right) \, d\xi},
\end{equation}
where $\xmax \in (0,1)$ is the unique solution to $\phi_1(\xmax)+\phi_2(\xmax)=T_2(\Lambda)$ ($=\phi_2(1)$).

\item
If the time $T$ satisfies \eqref{min time syst}, then the system \eqref{syst} is finite-time stabilizable with settling time $T$.
\end{enumerate}
\end{theorem}

Note in particular that the system \eqref{syst} is then null controllable in time $T$ if, and only if, it is finite-time stabilizable with settling time $T$.

\begin{example}
For $c$ satisfying the properties in \ref{ex 2} or given by the function in \ref{ex 3}, this result shows that the time $\Tunif$ cannot be improved.
This is not trivial, especially when $c$ is given by the function in \ref{ex 3}.
\end{example}

\begin{remark}\label{rem cst speeds}
When $\lambda_1, \lambda_2$ do not depend on space, the condition \eqref{min time syst}, in the situation $\Topt \leq T<\Tunif$, simply becomes
$$c=0 \quad \mbox{ in } \left(0, 1-\frac{T}{\Tunif}\right).$$
In particular, we see that we can possibly obtain any intermediate time between $\Topt$ and $\Tunif$.
Moreover, note that the value $\Topt$ is reachable even when $c$ is not identically equal to zero.
\end{remark}

\begin{remark}\label{rem naming}
As we shall see in the proof below, the most difficult part of this result is the necessary condition, that is the item \ref{item NC}.
It is also thanks to this part that we can call the time on the right-hand side of \eqref{min time syst} the minimal control time.
It is sometimes find in the literature that the time $\Tunif$ is ``the theoretical lower bound for control time'' or ``the optimal time''.
However, the failure of the controllability before this time is never proved in these works (and, in fact, it cannot be in general), which brings some confusion to our point of view.
This is why we carefully introduced a different naming and use the notations
\begin{itemize}
\item
$\Tunif$ as ``uniform time'', for the smallest time after which all the systems of the form \eqref{syst} are null controllable,
\item
$\Topt$ as ``optimal time'', for the smallest control time that can be obtained among all the possible control times for the systems of the form \eqref{syst}.
\end{itemize}
\end{remark}

\begin{remark}
Let us comment other possibilities for the boundary conditions at $x=0$:
\begin{enumerate}[(i)]
\item
When the boundary condition $y_2(t,0)=0$ is replaced by $y_2(t,0)=qy_1(t,0)$ with boundary coupling ``matrix'' $q \neq 0$, the result \cite[Theorem 3.2]{CVKB13} shows that the time $\Tunif$ is the minimal control time (more precisely, it is shown that the system \eqref{syst} with such a boundary condition is equivalent to the same system with no internal coupling matrix, for which $\Tunif$ is clearly minimal).
However, when $q=0$, we see that our time is smaller than the one obtained in this reference.

\item
When a second control is applied at the boundary $x=0$, i.e. the boundary condition $y_2(t,0)=0$ is replaced by $y_2(t,0)=v(t)$ with $v \in L^2_{\loc}(0,+\infty)$ a second control at our disposal, then the time $\Topt$ is the minimal control time.
The null controllability for $T \geq \Topt$ can be shown using for instance the well-known constructive method developed in \cite[Theorem 3.1]{Li10}.
On the other hand, the failure of the null controllability for $T<\Topt$ follows from the backstepping method (by means of Volterra transformation of the second kind) and a simple adaptation of Lemma \ref{Topt best} below.
\end{enumerate}
Therefore, combining the previous results of the literature with the new results of the present paper, we see that all the following possibilities for the boundary conditions have been handled:
\begin{gather*}
y_1(t,1)=p y_2(t,1)+ru(t), \quad y_2(t,0)=qy_1(t,0)+sv(t), \\
p,q,r,s \in \R \text{ with } (r,s) \neq (0,0).
\end{gather*}
\end{remark}

The rest of this article is organized as follows.
In Section \ref{sect backstepping}, we use the backstepping method to show that our initial system \eqref{syst} is equivalent to a canonical system from a controllability point of view.
In Section \ref{sect can} we use the Titchmarsh convolution theorem to  completely characterize the minimal control time for this canonical system.
In Section \ref{sect proof main} we characterize this time in terms of the parameters of the initial system.
Finally, in Section \ref{sect ext} we discuss possible extensions to systems with more than two equations.

\section{Reduction to a canonical form}\label{sect backstepping}

In this section, we perform some changes of unknown to transform our initial system \eqref{syst} into a new system whose controllability properties will be simpler to study, this is the so-called backstepping method for PDEs.
The content of section is quite standard by now, we refer for instance to \cite[Section 3.2]{CVKB13} for more details on the computations below.

First of all, we remove the diagonal terms in the system \eqref{syst}.
Using the invertible spatial transformation (seen as an operator from $L^2(0,1)^2$ onto itself)
\begin{equation}\label{transfo to tilde}
\begin{dcases}
\tilde{y}_1(t,x)=e_1(x)y_1(t,x), \\
\tilde{y}_2(t,x)=e_2(x)y_2(t,x),
\end{dcases}
\end{equation}
with
\begin{equation}\label{def ei}
e_1(x)=\exp\left(-\int_0^x \frac{a(\xi)}{\lambda_1(\xi)} \, d\xi\right),
\quad
e_2(x)=\exp\left(-\int_0^x \frac{d(\xi)}{\lambda_2(\xi)} \, d\xi\right),
\end{equation}
we easily see that the system \eqref{syst} is null controllable in time $T$ (resp. finite-time stabilizable with settling time $T$) if, and only if, so is the system
\begin{equation}\label{syst -diag}
\begin{dcases}
\pt{\tilde{y}_1}(t,x)+\lambda_1(x) \px{\tilde{y}_1}(t,x)=\tilde{b}(x)\tilde{y}_2(t,x), \\
\pt{\tilde{y}_2}(t,x)+\lambda_2(x) \px{\tilde{y}_2}(t,x)=\tilde{c}(x)\tilde{y}_1(t,x), \\
\begin{aligned}
& \tilde{y}_1(t,1)=\tilde{u}(t), && \tilde{y}_2(t,0)=0,  \\
& \tilde{y}_1(0,x)=\tilde{y}_1^0(x), && \tilde{y}_2(0,x)=\tilde{y}_2^0(x),
\end{aligned}
\end{dcases}
\quad t \in (0,+\infty), \, x \in (0,1),
\end{equation}
where
\begin{equation}\label{def ct}
\tilde{b}(x)=b(x) \frac{e_1(x)}{e_2(x)}
, \quad
\tilde{c}(x)=c(x) \frac{e_2(x)}{e_1(x)}.
\end{equation}

Let us now remove the coupling term on the first equation of \eqref{syst -diag} thanks to a second transformation.
Set
$$\Tau=\ens{(x,\xi) \in (0,1)\times(0,1) \quad \middle| \quad x>\xi}.$$

\noindent
Let $k_{11},k_{12},k_{21},k_{22} \in L^{\infty}(\Tau)$.
Using the spatial transformation
\begin{equation}\label{transfo to hat}
\begin{dcases}
\hat{y}_1(t,x)=\tilde{y}_1(t,x)-\int_0^x \left(k_{11}(x,\xi)\tilde{y}_1(t,\xi) +k_{12}(x,\xi)\tilde{y}_2(t,\xi)\right)\, d\xi, \\
\hat{y}_2(t,x)=\tilde{y}_2(t,x)-\int_0^x \left(k_{21}(x,\xi)\tilde{y}_1(t,\xi) +k_{22}(x,\xi)\tilde{y}_2(t,\xi)\right)\, d\xi,
\end{dcases}
\end{equation}
which is invertible since it is a Volterra transformation of the second kind (see e.g. \cite[Chapter 2, Theorem 5]{Hoc73}), we see that the system \eqref{syst -diag} is null controllable in time $T$ (resp. finite-time stabilizable with settling time $T$) if, and only if, so is the system
\begin{equation}\label{can syst}
\begin{dcases}
\pt{\hat{y}_1}(t,x)+\lambda_1(x) \px{\hat{y}_1}(t,x)=0, \\
\pt{\hat{y}_2}(t,x)+\lambda_2(x) \px{\hat{y}_2}(t,x)=g(x) \hat{y}_1(t,0), \\
\begin{aligned}
& \hat{y}_1(t,1)=\hat{u}(t), && \hat{y}_2(t,0)=0,  \\
& \hat{y}_1(0,x)=\hat{y}_1^0(x), && \hat{y}_2(0,x)=\hat{y}_2^0(x),
\end{aligned}
\end{dcases}
\quad t \in (0,+\infty), \, x \in (0,1),
\end{equation}
with $g$ given by
\begin{equation}\label{def g}
g(x)=- k_{21}(x,0)\lambda_1(0),
\end{equation}
provided that the kernels $k_{11}, k_{12}, k_{21}, k_{22}$ satisfy the so-called kernel equations:
\begin{equation}\label{kern equ 1}
\begin{dcases}
\lambda_1(x)\px{k_{11}}(x,\xi)+\pxi{k_{11}}(x,\xi)\lambda_1(\xi)
+k_{11}(x,\xi)\pxi{\lambda_1}(\xi)
+k_{12}(x,\xi)\tilde{c}(\xi)=0, \\
\lambda_1(x)\px{k_{12}}(x,\xi)+\pxi{k_{12}}(x,\xi)\lambda_2(\xi)
+k_{11}(x,\xi)\tilde{b}(\xi)
+k_{12}(x,\xi)\pxi{\lambda_2}(\xi)=0, \\
k_{11}(x,0)=0, \\
k_{12}(x,x)=\frac{\tilde{b}(x)}{\lambda_1(x)-\lambda_2(x)},
\end{dcases}
\quad (x,\xi) \in \Tau,
\end{equation}
and
\begin{equation}\label{kern equ 2}
\begin{dcases}
\lambda_2(x)\px{k_{21}}(x,\xi)+\pxi{k_{21}}(x,\xi)\lambda_1(\xi)
+k_{21}(x,\xi)\pxi{\lambda_1}(\xi)
+k_{22}(x,\xi)\tilde{c}(\xi)=0, \\
\lambda_2(x)\px{k_{22}}(x,\xi)+\pxi{k_{22}}(x,\xi)\lambda_2(\xi)
+k_{21}(x,\xi)\tilde{b}(\xi)
+k_{22}(x,\xi)\pxi{\lambda_2}(\xi)=0, \\
k_{21}(x,x)=\frac{\tilde{c}(x)}{\lambda_2(x)-\lambda_1(x)},
\end{dcases}
\quad (x,\xi) \in \Tau.
\end{equation}
Note that \eqref{kern equ 1} and \eqref{kern equ 2} are not coupled.

From \cite[Theorem A.1]{CVKB13}, we know that the kernel equations \eqref{kern equ 1}-\eqref{kern equ 2} have a solution.
More precisely, we have the following result:
\begin{theorem}\label{thm kern}
For every $k^0 \in L^{\infty}(0,1)$, there exists a unique solution $(k_{11}, k_{12}, k_{21}, k_{22}) \in L^{\infty}(\Tau)^4$ to the kernel equations \eqref{kern equ 1}-\eqref{kern equ 2} with
$$k_{22}(x,0)=k^0(x), \quad x \in (0,1).$$
\end{theorem}

In the aforementioned reference this result is stated in a $C^0$ framework (assuming that $a,b,c,d \in C^0([0,1])$) but its proof readily shows that it is valid in $L^{\infty}$ as well.
As before, the notion of solution is to be understood in the sense of solution along the characteristics.
The boundary terms such as $k_{21}(x,0)$, which defines $g$ (see \eqref{def g}), or $k_{11}(1,\xi), k_{12}(1,\xi)$, that will appear shortly below in our feedback law (see \eqref{def feedback}), etc. are also understood in this sense.
We refer for instance to the formula \eqref{def trace k21} below for the precise meaning of $k_{21}(x,0)$.

\section{Study of the canonical system}\label{sect can}

We call the system \eqref{can syst} the ``control canonical form of the system \eqref{syst}'' or ``canonical system'' in short, by analogy with \cite{Bru70,Rus78-JMAA} and since we will see in this section that we are able to directly read its controllability properties (a task that seems impossible on the initial system \eqref{syst}).

The goal of this section is to establish the following result:

\begin{theorem}\label{thm can syst}
Let $T>0$ and $g \in L^{\infty}(0,1)$.
\begin{enumerate}[(i)]
\item\label{item NC can syst}
If the system \eqref{can syst} is null controllable in time $T$, then necessarily
\begin{equation}\label{min time can syst}
T \geq \max\ens{T_1(\Lambda)+\int_{\Xm{1}{g}}^1 \frac{1}{\lambda_2(\xi)} \,d\xi, \quad T_2(\Lambda)}.
\end{equation}

\item\label{item SC can syst}
If the time $T$ satisfies \eqref{min time can syst}, then the system \eqref{can syst} is finite-time stable with settling time $T$.
\end{enumerate}
\end{theorem}

Let us emphasize once again that the difficult point is the first item.

\begin{remark}
Since $\hat{u}=0$ stabilizes the canonical system \eqref{can syst} by \ref{item SC can syst} of Theorem \ref{thm can syst}, we see from the formula \eqref{transfo to tilde} and \eqref{transfo to hat} that our feedback for the system \eqref{syst} is then
\begin{equation}\label{def feedback}
u(t)=\int_0^1 \frac{k_{11}(1,\xi)e_1(\xi)}{e_1(1)} y_1(t,\xi)\,d\xi
+\int_0^1 \frac{k_{12}(1,\xi)e_2(\xi)}{e_1(1)} y_2(t,\xi)\,d\xi.
\end{equation}
Note that $u \in C^0([0,+\infty))$.
\end{remark}

\subsection{The characteristics}\label{sect caract}

Before proving Theorem \ref{thm can syst} we need to introduce the characteristic curves associated with the system \eqref{can syst} and recall some useful properties.

First of all, it is convenient to extend $\lambda_1,\lambda_2$ to functions of $\R$ (still denoted by the same) such that $\lambda_1,\lambda_2 \in C^{0,1}(\R)$ and 
\begin{equation}\label{hyp speeds bis}
\lambda_1(x) \leq -\epsilon<0<\epsilon<\lambda_2(x), \quad \forall x \in \R,
\end{equation}
for some $\epsilon>0$ small enough.
Since all the results of the present paper depend only on the values of $\lambda_1,\lambda_2$ in $[0,1]$, they do not depend on such an extension.

In what follows, $i \in \ens{1,2}$.
Let $\chi_i$ be the flow associated with $\lambda_i$, i.e. for every $(t,x) \in \R \times \R$, the function $s \mapsto \chi_i(s;t,x)$ is the solution to the ODE
\begin{equation}\label{ODE xi}
\begin{dcases}
\pas{\chi_i}(s;t,x)=\lambda_i(\chi_i(s;t,x)), \quad \forall s \in \R, \\
\chi_i(t;t,x)=x.
\end{dcases}
\end{equation}
The existence and uniqueness of a (global) solution to the ODE \eqref{ODE xi} follows from the (global) Cauchy-Lipschitz theorem (see e.g. \cite[Theorem II.1.1]{Har02}).
The uniqueness also yields the important group property
\begin{equation}\label{chi invariance}
\chi_i\left(\sigma;s,\chi_i(s;t,x)\right)=\chi_i(\sigma;t,x), \quad \forall \sigma,s \in \R.
\end{equation}
By classical regularity results on ODEs (see e.g. \cite[Theorem V.3.1]{Har02}), we have $\chi_i \in C^1(\R^3)$ and
\begin{equation}\label{chi increases}
\pt{\chi_i}(s;t,x)=-\lambda_i(\chi_i(s;t,x)), \quad \px{\chi_i}(s;t,x)=\frac{\lambda_i(\chi_i(s;t,x))}{\lambda_i(x)}.
\end{equation}

Let us now introduce the entry and exit times $\ssin_i(t,x),\ssout_i(t,x) \in \R$ of the flow $\chi_i(\cdot;t,x)$ inside the domain $[0,1]$, i.e. the respective unique solutions to
$$
\begin{dcases}
\chi_1(\ssin_1(t,x);t,x)=1, \quad \chi_1(\ssout_1(t,x); t, x)=0, \\
\chi_2(\ssin_2(t,x);t,x)=0, \quad \chi_2(\ssout_2(t,x);t,x)=1.
\end{dcases}
$$
Their existence and uniqueness are guaranteed by the condition \eqref{hyp speeds bis}.
It readily follows from \eqref{chi invariance} and the uniqueness of $\ssin_i$ that
\begin{equation}\label{invariance}
\ssin_i\left(s,\chi_i(s;t,x)\right)=\ssin_i(t,x), \quad \forall s \in \R.
\end{equation}

\noindent
By the implicit function theorem we have $\ssin_i \in C^1(\R^2)$ with (using \eqref{chi increases})
\begin{equation}\label{ssin monotonicity}
\begin{dcases}
\pt{\ssin_1}(t,x)>0, \quad \px{\ssin_1}(t,x)>0, \\
\pt{\ssin_2}(t,x)>0, \quad \px{\ssin_2}(t,x)<0.
\end{dcases}
\end{equation}

\noindent
Combined with the group property \eqref{invariance}, this yields the following inverse formula for every $s,t \in \R$:
\begin{equation}\label{inv sin}
\begin{dcases}
s<\ssout_1(t,1) \quad \Longleftrightarrow \quad \ssin_1(s,0)<t, \\
s<\ssout_2(t,0) \quad \Longleftrightarrow \quad \ssin_2(s,1)<t.
\end{dcases}
\end{equation}

Finally, since $\lambda_i$ does not depend on time, we have an explicit formula for the inverse function $\theta \mapsto \chi_i^{-1}(\theta;t,x)$.
Indeed, it solves
$$
\begin{dcases}
\frac{\partial (\chi_i^{-1})}{\partial \theta}(\theta;t,x)=\frac{1}{\pas{\chi_i}\left(\chi_i^{-1}(\theta;t,x);t,x\right)}=\frac{1}{\lambda_i(\theta)}, \quad \forall \theta \in \R,\\
\chi_i^{-1}(x;t,x)=t,
\end{dcases}
$$
which gives
\begin{equation}\label{chi inv}
\chi_i^{-1}(\theta;t,x)=t+\int_x^\theta \frac{1}{\lambda_i(\xi)} \, d\xi.
\end{equation}

\noindent
This also yields an explicit formula for $\ssin_1,\ssin_2$ and $\ssout_1,\ssout_2$ and, in particular,
$$T_1(\Lambda)=\ssout_1(0,1), \quad T_2(\Lambda)=\ssout_2(0,0).$$

\subsection{Proof of Theorem \ref{thm can syst}}

First of all, the solution of the canonical system \eqref{can syst} is explicitly given by:
\begin{equation}\label{explicit formula 1}
\hat{y}_1(t,x)=
\begin{dcases}
\hat{y}^0_1\left(\chi_1(0;t,x)\right)
& \mbox{ if } \ssin_1(t,x)<0, \\
\hat{u}\left(\ssin_1(t,x)\right)
& \mbox{ if } \ssin_1(t,x)>0,
\end{dcases}
\end{equation}
and
\begin{equation}\label{explicit formula 2}
\hat{y}_2(t,x)=
\begin{dcases}
\hat{y}^0_2\left(\chi_2(0;t,x)\right)
+\int_0^t g\left(\chi_2(s;t,x)\right) \hat{y}_1(s,0) \, ds
& \mbox{ if } \ssin_2(t,x)<0, \\
\int_{\ssin_2(t,x)}^t g\left(\chi_2(s;t,x)\right) \hat{y}_1(s,0) \, ds
& \mbox{ if } \ssin_2(t,x)>0.
\end{dcases}
\end{equation}

Next, we show a uniform lower bound for the control time:

\begin{lemma}\label{Topt best}
Let $T>0$.
If the system \eqref{can syst} is null controllable in time $T$, then necessarily
$$T \geq \Topt.$$
\end{lemma}

This result states that the control time cannot be better than the one of the case $g=0$.

\begin{proof}
For $i \in \ens{1,2}$, let $\omega_i$ be the open subset defined by
$$\omega_i=\ens{x \in (0,1) \quad \middle| \quad \ssin_i(T,x)<0}.$$
From \eqref{inv sin} and \eqref{ssin monotonicity}, we see that
\begin{equation}\label{cns T>T_i}
T \geq T_i(\Lambda)
\quad \Longleftrightarrow \quad
\omega_i=\emptyset.
\end{equation}
Therefore, if $T<T_1(\Lambda)$, then we see from \eqref{explicit formula 1} that $\hat{y}_1^0$ can be chosen so that $\hat{y}_1(T,x) \neq 0$ for $x \in \omega_1$, whatever $\hat{u}$ is.
On the other hand, if $T<T_2(\Lambda)$ and if the system \eqref{can syst} is null controllable in time $T$, then for every $\hat{y}^0_2 \in L^2(0,1)$, there exists $\hat{u} \in L^2(0,T)$ such that, for a.e. $x \in \omega_2$, we have
$$
0=
\hat{y}^0_2\left(\chi_2(0;T,x)\right)
+\int_0^T g\left(\chi_2(s;T,x)\right) \hat{y}_1(s,0) \, ds.
$$

\noindent
Since $x \in \omega_2 \mapsto \chi_2(0;T,x)$ is bijective (it is increasing by \eqref{chi increases} and $\omega_2$ is an interval by \eqref{ssin monotonicity}), this implies that the bounded linear operator $K:L^2(0,T) \longrightarrow L^2(\omega_2)$ defined by
$$
(Kh)(x)
=-\int_0^T g\left(\chi_2(s;T,x)\right) h(s) \, ds,
$$
is surjective.
This is impossible since its range is clearly a subset of $L^{\infty}(\omega_2)$, which is a proper subset of $L^2(\omega_2)$ (alternatively, one could note that $K$ is compact and therefore it cannot be surjective over an infinite dimensional space, see e.g. \cite[Theorem 4.18 (b)]{Rud91}).

\end{proof}

The proof of the item \ref{item NC can syst} of Theorem \ref{thm can syst} crucially relies on the Titchmarsh convolution theorem \cite[Theorem VII]{Tit26} (see also \cite[Chapter XV]{Mik78}):
\begin{theorem}\label{TT thm}
Let $\alpha,\beta \in L^1(0,\bar{\tau})$ ($\bar{\tau}>0$).
We have
\begin{equation}\label{conv zero}
\int_0^{\tau} \alpha(\tau-\sigma) \beta(\sigma) \, d\sigma=0, \quad \mbox{ a.e. } 0<\tau<\bar{\tau},
\end{equation}
if, and only if,
$$\Xm{\bar{\tau}}{\alpha}+\Xm{\bar{\tau}}{\beta} \geq \bar{\tau}.$$
\end{theorem}

\begin{remark}
The difficulty in the proof of this result is the necessary condition, i.e. the implication ``$\Longrightarrow$'', just like it is the case for our main result.
Let us however mention that its proof is easy in case $\alpha$ satisfies the condition in \ref{ex 2} of Example \ref{examples} (by taking derivatives of \eqref{conv zero} and using the injectivity of Volterra transformations of the second kind).
It does not seem trivial for functions of the form \eqref{bump f} though.
\end{remark}

We are now ready to prove the main result of Section \ref{sect can}:

\begin{proof}[Proof of Theorem \ref{thm can syst}]
~
\begin{enumerate}[1)]
\item
Thanks to Lemma \ref{Topt best}, we can assume that $T \geq T_1(\Lambda)$ and $T \geq T_2(\Lambda)$.
This means that $\ssin_1(T,x)>0$ and $\ssin_2(T,x)>0$ for every $x \in (0,1)$ (see \eqref{cns T>T_i} and \eqref{ssin monotonicity}).
It then follows from the explicit formula \eqref{explicit formula 1} and \eqref{explicit formula 2} that $\hat{y}_1(T,\cdot)=0$ if, and only if,
\begin{equation}\label{caract yu zero}
\hat{u}\left(\ssin_1(T,x)\right)=0, \quad 0<x<1,
\end{equation}
and $\hat{y}_2(T,\cdot)=0$ if, and only if,
\begin{equation}\label{caract yd zero}
\int_{\ssin_2(T,x)}^T g\left(\chi_2(s;T,x)\right) \hat{y}_1(s,0) \, ds=0, \quad 0<x<1.
\end{equation}

\item\label{step convo}
Let us focus on the second condition \eqref{caract yd zero}.
Writing $x=\chi_2(T;t,0)$, which belongs to $(0,1)$ for $t \in (\ssin_2(T,1),T)$ (recall in particular \eqref{inv sin}), and using the group properties \eqref{chi invariance} and \eqref{invariance} with the identity $\ssin_2(t,0)=t$, we obtain that $\hat{y}_2(T,\cdot)=0$ if, and only if,
\begin{equation}\label{NC g}
\int_t^T g\left(\chi_2(s;t,0)\right) \hat{y}_1(s,0) \, ds=0, \quad \ssin_2(T,1)<t<T.
\end{equation}

Now we use the fact that $g\left(\chi_2(s;t,0)\right)$ is actually a function of $s-t$.
Indeed, by uniqueness to the solution to the ODE \eqref{ODE xi}, we see that the characteristics take the form
$$\chi_i(s;t,x)=\tilde{\chi}_i(s-t;x),$$
where $s \mapsto \tilde{\chi}_i(s;x)$ is the unique solution to
$$
\begin{dcases}
\pas{\tilde{\chi}_i}(s;x)=\lambda_i(\tilde{\chi}_i(s;x)), \quad \forall s \in \R, \\
\tilde{\chi}_i(0;x)=x.
\end{dcases}
$$

Using the change of variables $\sigma=s-t$ and introducing
$$\alpha(\theta)=\hat{y}_1(-\theta+T,0),
\qquad
\beta(\theta)=g\left(\tilde{\chi}_2(\theta;0)\right),
\quad 0<\theta<T-\ssin_2(T,1),$$
we see that \eqref{NC g} is equivalent to (setting $\tau=T-t$)
\begin{equation}\label{cns yd}
\int_0^{\tau} \alpha(\tau-\sigma) \beta(\sigma) \, d\sigma=0,
\quad 0<\tau<\bar{\tau},
\end{equation}
where
$$\bar{\tau}=T-\ssin_2(T,1).$$

\item
Applying the Titchmarsh convolution theorem (Theorem \ref{TT thm}) we deduce that \eqref{cns yd} is equivalent to
$$\Xm{\bar{\tau}}{\alpha}+\Xm{\bar{\tau}}{\beta} \geq \bar{\tau}.$$

From the explicit expression \eqref{explicit formula 1} and the inverse formula \eqref{inv sin}, we see that
$$
\alpha(\theta)=
\begin{dcases}
\hat{y}^0_1\left(\chi_1(0;-\theta+T,0)\right)
& \mbox{ if } \theta>T-\ssout_1(0,1), \\
\hat{u}\left(\ssin_1(-\theta+T,0)\right)
& \mbox{ if } \theta<T-\ssout_1(0,1).
\end{dcases}
$$
Therefore, we can choose $\hat{y}_1^0$ so that
$$
\alpha(\theta) \neq 0, \quad \forall \theta \in
\left(T-\ssout_1(0,1), T-\ssout_1(0,1)+\epsilon\right),
$$
for some $0<\epsilon<\ssout_1(0,1)$.
This yields the bound
$$\Xm{\bar{\tau}}{\alpha} \leq T-\ssout_1(0,1).$$
Consequently, we necessarily have
\begin{equation}\label{cond m=1}
\Xm{\bar{\tau}}{\beta} \geq \ssout_1(0,1)-\ssin_2(T,1).
\end{equation}

Since $s \mapsto \tilde{\chi}_2(s;0)$ is increasing with $\tilde{\chi}_2(0;0)=0$, this is equivalent to
\begin{multline*}
\Xm{1}{g} \geq \tilde{\chi}_2\left(\ssout_1(0,1)-\ssin_2(T,1);0\right)
=\chi_2\left(\ssout_1(0,1); \ssin_2(T,1), 0\right)
\\
=\chi_2\left(\ssout_1(0,1); T, 1\right) \quad (\text{by \eqref{chi invariance} with $s=\ssin_2(T,1)$}).
\end{multline*}

Since $s \mapsto \chi_2(s;T,1)$ is increasing, this is also equivalent to
$$\chi_2^{-1}(\Xm{1}{g};T,1) \geq \ssout_1(0,1)=T_1(\Lambda).$$
Using the explicit expression \eqref{chi inv}, we then obtain the desired condition $T \geq T_1(\Lambda)+\int_{\Xm{1}{g}}^1 \frac{1}{\lambda_2(\xi)} \,d\xi$.

\item
Conversely, assume that $T$ satisfies this condition and $T \geq T_2(\Lambda)$.
Then, \eqref{cond m=1} holds by the previous equivalences.
Taking $\hat{u}=0$, we see that $\alpha=0$ in $\left(0,T-T_1(\Lambda)\right)$, which yields
$$\Xm{\bar{\tau}}{\alpha}+\Xm{\bar{\tau}}{\beta} \geq T-T_1(\Lambda)+\ssout_1(0,1)-\ssin_2(T,1)= T-\ssin_2(T,1)=\bar{\tau}.$$
This implies \eqref{cns yd} (here we only use the ``easy part'' of the Titchmarsh convolution theorem) and thus $\hat{y}_2(T,\cdot)=0$.
Finally, note that $\hat{u}=0$ also obviously satisfies \eqref{caract yu zero} and thus $\hat{y}_1(T,\cdot)=0$ as well.

\end{enumerate}

\end{proof}

\begin{remark}
Let us point out that the space dependence of the speeds brings up more technical difficulties than the case of constant speeds (especially the step \ref{step convo}).
\end{remark}

\section{Proof of the main result}\label{sect proof main}

In this section we show how to deduce our main result from Theorem \ref{thm can syst}.

\begin{proof}[Proof of Theorem \ref{main thm}]
~
\begin{enumerate}[1)]
\item
First of all, let us recall that the initial system \eqref{syst} is null controllable in time $T$ (resp. finite-time stabilizable with settling time $T$) if, and only if, so is the canonical system \eqref{can syst} (with $g$ given by \eqref{def g}).
Therefore, thanks to Theorem \ref{thm can syst} it suffices to show that
$$
T \geq T_1(\Lambda)+\int_{\Xm{1}{g}}^1 \frac{1}{\lambda_2(\xi)} \,d\xi
\quad \Longleftrightarrow \quad
T \geq \int_{\Xm{\xmax}{c}}^1 \left(\frac{1}{-\lambda_1(\xi)}+\frac{1}{\lambda_2(\xi)}\right) \, d\xi,
$$
which amounts to characterize $\Xm{1}{g}$ in terms of $\Xm{\xmax}{c}$ (we recall that $\xmax$ is defined in the statement of Theorem \ref{main thm}).
To this end, we are going to prove the identity
\begin{equation}\label{caract xg xc}
\phi_2(\Xm{1}{g})=\phi_1(\Xm{\xmax}{c})+\phi_2(\Xm{\xmax}{c}),
\end{equation}
where we recall that $\phi_1,\phi_2 \in C^{1,1}([0,1])$ are defined in \eqref{def phii}.

\item
We recall that $g(x)=-k_{21}(x,0)\lambda_1(0)$, where $k_{21}$ is the solution in $\Tau$ to
\begin{equation}\label{kern equ anew}
\begin{dcases}
\lambda_2(x)\px{k_{21}}(x,\xi)+\pxi{k_{21}}(x,\xi)\lambda_1(\xi)
+k_{21}(x,\xi)\pxi{\lambda_1}(\xi)
+k_{22}(x,\xi)\tilde{c}(\xi)=0, \\
k_{21}(x,x)=\frac{\tilde{c}(x)}{\lambda_2(x)-\lambda_1(x)},
\end{dcases}
\end{equation}
and where $\tilde{c}$ is defined in \eqref{def ct} and \eqref{def ei} (note that $\Xm{\epsilon}{\tilde{c}}=\Xm{\epsilon}{c}$ for any $\epsilon \in (0,1]$).
Let $s \mapsto \chi(s;x)$ be the associated characteristic passing through $(x,\xi)=(x,0)$, i.e. the solution to the ODE
\begin{equation}\label{ODE chi21}
\begin{dcases}
\pas{\chi}(s;x)=\frac{\lambda_1(\chi(s;x))}{\lambda_2(s)}, \quad \forall s \in \R, \\
\chi(x;x)=0,
\end{dcases}
\end{equation}
(we recall that $\lambda_1,\lambda_2$ have been extended to $\R$ in Section \ref{sect caract}).
We have $\chi \in C^1(\R^2)$ by classical regularity results on ODEs with
$$
\px{\chi}(s;x)=\frac{-\lambda_1(\chi(s;x))}{\lambda_2(x)}>0.
$$
Since $f:s \mapsto s-\chi(s;x)$ is continuous and increasing with $\lim_{s \to \mp \infty} f(s)=\mp \infty$, there exists a unique solution $\ssin(x) \in \R$ to
$$\chi\left(\ssin(x);x\right)=\ssin(x).$$
Besides, for every $x \in(0,1)$, we have $0<\ssin(x)<x$ and
$$(s,\chi(s;x)) \in \Tau, \quad \forall s \in (\ssin(x),x).$$
By the implicit function theorem we have $\ssin \in C^1(\R)$ with, for every $x \in \R$,
\begin{equation}\label{ssin prime}
(\ssin)'(x)=\frac{\px{\chi}(\ssin(x);x)}{1-\pas{\chi}(\ssin(x);x)}>0.
\end{equation}

In particular, the inverse function $(\ssin)^{-1}:[0,\ssin(1)] \longrightarrow [0,1]$ exists.
We are going to show that
\begin{equation}\label{xg ssin}
\Xm{\ssin(1)}{\tilde{c}}=\ssin(\Xm{1}{g}).
\end{equation}

Along the characteristics, the solution to \eqref{kern equ anew} satisfies, for $s \in \left(\ssin(x),x\right)$,
$$
\begin{dcases}
\frac{d}{ds} k_{21}(s,\chi(s;x))=
\frac{-\pxi{\lambda_1}(\chi(s;x))}{\lambda_2(s)}k_{21}(s,\chi(s;x))
+\frac{-k_{22}(s,\chi(s;x))}{\lambda_2(s)}\tilde{c}(\chi(s;x)), \\
k_{21}\left(\ssin(x),\ssin(x)\right)=\frac{\tilde{c}(\ssin(x))}{\lambda_2(\ssin(x))-\lambda_1(\ssin(x))}.
\end{dcases}
$$

Consequently,
\begin{equation}\label{def trace k21}
k_{21}(x,0)=r(x)\tilde{c}(\ssin(x))+\int_{\ssin(x)}^x h(x,\sigma) \tilde{c}\left(\chi(\sigma;x)\right) \,d\sigma,
\end{equation}
with
$$
r(x)=\exp\left(
\int_{\ssin(x)}^x \frac{-\pxi{\lambda_1}(\chi(s;x))}{\lambda_2(s)} \, ds
\right) \frac{1}{\lambda_2(\ssin(x))-\lambda_1(\ssin(x))},
$$
and
$$
h(x,\sigma)=
\exp\left(
\int_{\sigma}^x \frac{-\pxi{\lambda_1}(\chi(s;x))}{\lambda_2(s)} \, ds
\right)
\frac{-k_{22}(\sigma,\chi(\sigma;x))}{\lambda_2(\sigma)}.
$$

Using the change of variable $\theta=(\ssin)^{-1}(\chi(\sigma;x))$, we obtain
$$\frac{1}{r(x)}
k_{21}(x,0)=\tilde{c}(\ssin(x))+\int_{0}^{x} \tilde{h}(x,\theta) \tilde{c}\left(\ssin(\theta)\right) \,d\theta,$$
with kernel
$$
\tilde{h}(x,\theta)=
\frac{1}{r(x)}
h\left(x,\chi^{-1}(\ssin(\theta);x)\right)
\frac{(\ssin)'(\theta)}{\pas{\chi}(\chi^{-1}(\ssin(\theta);x);x)}.
$$

We can check that $\tilde{h} \in L^{\infty}(\Tau)$ (recall \eqref{ssin prime}).
It follows from the injectivity of Volterra transformations of the second kind that
$$\Xm{1}{g}=\Xm{1}{\tilde{c} \,\circ \ssin},$$
which is equivalent to \eqref{xg ssin} since $\ssin$ is increasing with $\ssin(0)=0$.

\item
To conclude the proof, it remains to observe that the solution to the ODE \eqref{ODE chi21} satisfies
$$\phi_1(\chi(s;x))=\phi_2(x)-\phi_2(s),$$
for every $x \in [0,1]$ and $s \in [\ssin(x),x]$.
Taking $x=1$ and $s=\ssin(1)$, we see that $\ssin(1)=\xmax$ (by uniqueness of the solution to the equation $\phi_1(\xmax)+\phi_2(\xmax)=\phi_2(1)$).
Taking then $x=\Xm{1}{g}$ and $s=\ssin(\Xm{1}{g})=\Xm{\xmax}{\tilde{c}}$ (recall \eqref{xg ssin}), we obtain the desired identity \eqref{caract xg xc}.

\end{enumerate}
\end{proof}

\begin{remark}
In the proof of Theorem \ref{main thm}, we have not used the apparent freedom for the boundary data of $k_{22}$ provided by Theorem \ref{thm kern}.
\end{remark}

\section{Extensions and open problems}\label{sect ext}

The results of this paper can be partially extended to systems of more than $2$ equations.
More precisely, we can consider the following $n \times n$ systems ($n \geq 2$):

\begin{equation}\label{syst m=1}
\begin{dcases}
\pt{y_1}(t,x)+\lambda_1(x) \px{y_1}(t,x)=a(x)y_1(t,x)+B(x)y_+(t,x), \\
\pt{y_+}(t,x)+\Lambda_+(x) \px{y_+}(t,x)=C(x)y_1(t,x)+D(x)y_+(t,x), \\
\begin{aligned}
& y_1(t,1)=u(t), && y_+(t,0)=Qy_1(t,0),  \\
& y_1(0,x)=y_1^0(x), && y_+(0,x)=y_+^0(x),
\end{aligned}
\end{dcases}
\quad t \in (0,+\infty), \, x \in (0,1).
\end{equation}
In \eqref{syst m=1}, $(y_1(t,\cdot),y_+(t,\cdot)) \in \R \times \R^{n-1}$ is the state at time $t$, $(y_1^0,y_+^0)$ is the initial data and $u(t) \in \R$ is the control at time $t$.
We assume that we have one negative speed $\lambda_1 \in C^{0,1}([0,1])$ and $n-1$ positive speeds $\lambda_2,\ldots,\lambda_n \in C^{0,1}([0,1])$ such that:
\begin{equation}\label{hyp speeds m=1}
\lambda_1(x)<0<\lambda_2(x)<\cdots<\lambda_n(x), \quad \forall x \in [0,1],
\end{equation}
and we use the notation $\Lambda_+=\diag(\lambda_2,\ldots,\lambda_n)$.
Finally, $a \in L^{\infty}(0,1)$, $B \in L^{\infty}(0,1)^{1 \times (n-1)}$, $C \in L^{\infty}(0,1)^{n-1}$, $D \in L^{\infty}(0,1)^{(n-1) \times (n-1)}$ couple the equations of the system inside the domain and the constant matrix $Q \in \R^{n-1}$ couples the equations of the system on the boundary $x=0$.

Let us now introduce the times defined by
$$
T_1(\Lambda)=\int_0^1 \frac{1}{-\lambda_1(\xi)} \, d\xi,
\quad
T_i(\Lambda)=\int_0^1 \frac{1}{\lambda_i(\xi)} \, d\xi, \quad \forall i \in \ens{2,\ldots,n}.
$$
Note that $T_n(\Lambda)<\ldots<T_2(\Lambda)$ by \eqref{hyp speeds m=1}.

It was established in \cite{DMVK13} and \cite[Lemma 3.1]{HDMVK16} that the system \eqref{syst m=1} is finite-time stabilizable with setting time $T$ if $T \geq \Tunif$, where $\Tunif$ is still given by \eqref{def Topt Tunif}.

Using the backstepping method (see e.g. \cite[Section 2.2]{HVDMK19}), it can be shown as before that the system \eqref{syst m=1} is null controllable in time $T$ (resp. finite-time stabilizable with settling time $T$) if, and only if, so is the system
\begin{equation}\label{can syst m=1}
\begin{dcases}
\pt{\hat{y}_1}(t,x)+\lambda_1(x) \px{\hat{y}_1}(t,x)=0, \\
\pt{\hat{y}_+}(t,x)+\Lambda_+(x) \px{\hat{y}_+}(t,x)=G(x) \hat{y}_1(t,0), \\
\begin{aligned}
& \hat{y}_1(t,1)=\hat{u}(t), && \hat{y}_+(t,0)=Q\hat{y}_1(t,0),  \\
& \hat{y}_1(0,x)=\hat{y}_1^0(x), && \hat{y}_+(0,x)=\hat{y}_+^0(x),
\end{aligned}
\\
\end{dcases}
\quad t \in (0,+\infty), \, x \in (0,1),
\end{equation}
for some $G \in L^{\infty}(0,1)^{n-1}$ depending on all the parameters $\lambda_1,\Lambda_+$, $a,B,C,D$ and $Q$.

By mimicking the proof of Theorem \ref{thm can syst}, we can obtain the following result:

\begin{theorem}\label{thm m=1}
Let $T>0$.
\begin{enumerate}[(i)]
\item
If the system \eqref{can syst m=1} is null controllable in time $T$, then necessarily
\begin{equation}\label{min time can syst m=1}
T \geq \max\ens{T_1(\Lambda)+\max_{i \in \ens{2,\ldots,n}} T(\lambda_i,g_{i-1},q_{i-1}), \quad T_2(\Lambda)},
\end{equation}
where
$$
T(\lambda_i,g_{i-1},q_{i-1})=
\begin{dcases}
\int_{\Xm{1}{g_{i-1}}}^1 \frac{1}{\lambda_i(\xi)} \,d\xi & \mbox{ if } q_{i-1}=0, \\
T_i(\Lambda) & \mbox{ if } q_{i-1} \neq 0.
\end{dcases}
$$

\item
If the time $T$ satisfies \eqref{min time can syst m=1}, then the system \eqref{can syst m=1} is finite-time stable with settling time $T$.
\end{enumerate}
\end{theorem}

However, we are unable so far to deduce from this result some explicit condition for the initial system \eqref{syst m=1}.
The main technical problem is that $G$ is heavily coupled on the parameters $\lambda_1,\Lambda_+$, $a,B,C,D$ and $Q$ (see e.g. \cite[Section 2.2]{HVDMK19}).
We leave it as an open problem that could be investigated in future works.

\section*{Acknowledgements}

The first author would like to thank Institute of Mathematics in Jagiellonian University for its hospitality.
This work was initiated while he was visiting there.
This project was supported by National Natural Science Foundation of China (No. 12071258) and the Young Scholars Program of Shandong University (No. 2016WLJH52).


\bibliographystyle{amsalpha}
\bibliography{biblio}

\end{document}